\documentclass[preprint,11pt]{elsarticle}
\usepackage{lmodern}
\usepackage{amsmath,amssymb,amsthm}
\usepackage[margin=1in]{geometry}
\usepackage{microtype}
\usepackage{hyperref}
\usepackage{xcolor}

\numberwithin{equation}{section}
\newtheorem{theorem}{Theorem}[section]
\newtheorem{lemma}[theorem]{Lemma}
\newtheorem{proposition}[theorem]{Proposition}
\newtheorem{corollary}[theorem]{Corollary}

\theoremstyle{definition}
\newtheorem{definition}[theorem]{Definition}

\newtheorem{remark}[theorem]{Remark}

\newcommand{\R}{\mathbb{R}}

\newcommand{\W}{\mathcal{W}}
\newcommand{\supp}{\operatorname{supp}_{+}}
\newcommand{\TP}{\mathrm{TP}_{2}}
\newcommand{\RR}{\mathrm{RR}_{2}}
\newcommand{\fall}[2]{(#1)_{\underline{#2}}}
\newcommand{\transpose}{\mathsf{T}}

\allowdisplaybreaks[4]
\begin{document}

\begin{frontmatter}
\title{Preservation of log-concavity under Hadamard products}
\author[]{Yanxin Liu}
\ead{liu-yanxin@outlook.com}

\author[]{Jianxi Mao\corref{cor1}}
\ead{maojx@dlut.edu.cn}
\cortext[cor1]{Corresponding author}

\address{School of Mathematical Sciences, Dalian University of Technology, Dalian 116024,\\ P. R. China}

\begin{abstract}
For a nonzero real polynomial $p$, let $\W(p)$ denote the numerator of its ordinary generating function. 
We prove that if the coefficients of both $\W(p)$ and $\W(q)$ are nonnegative and log-concave with no internal zeros, then so are the coefficients of $\W(pq)$. 
This provides an affirmative answer to a question of Br\"and\'en, Ferroni, and Jochemko. 
As applications, 
we derive corresponding results for finite products and for Cartesian products of lattice polytopes, answering a question of Ferroni and Higashitani.
\end{abstract}


\end{frontmatter}
\section{Introduction}
Let $p\in\mathbb{R}[t]$ be a nonzero polynomial of degree $d$. There is a unique polynomial $\W(p)$ of degree at most $d$ such that
\begin{equation}\label{eq:W-definition}
\sum_{n\geq 0} p(n)x^n=\frac{\W(p)(x)}{(1-x)^{d+1}}.
\end{equation}
If $q$ is a polynomial of degree $d_1$, then the generating function of the Hadamard product of the two sequences $(p(n))$ and $(q(n))$ 
is given by
\begin{equation*}
\sum_{n\geq0}p(n)q(n)x^n
=\frac{\W(pq)(x)}{(1-x)^{d+d_1+1}}.
\end{equation*}

The Hadamard product plays an important role in the theory of formal power series, as it frequently encodes natural structural constructions~\cite{AF11,Sta12,Wag92}.
A classic example comes from Ehrhart theory: 
the Hadamard product of the Ehrhart series of two lattice polytopes corresponds to the Ehrhart series of the Cartesian product of the two polytopes ~\cite{BR15,FH24}. 
More generally, in commutative algebra, the Hadamard product of the Hilbert series of two standard graded algebras is the Hilbert series of their Segre product (see, for instance,~\cite{FK14}).
In this paper, we study log-concavity for sequences with no internal zeros.
\begin{definition}\label{def:LC}
A nonnegative sequence $w=(w_0,\ldots,w_d)$
is \emph{log-concave with no internal zeros} (or LC-NIZ) if
\begin{equation*}
w_i^2\geq w_{i-1}w_{i+1}\qquad \textrm{for }\, 1\le i \le d-1,
\end{equation*}
and its positive support
$\supp(w)=\{i:w_i>0\}$ is an integer interval.
A polynomial is said to be LC-NIZ when its
coefficient sequence is LC-NIZ.
\end{definition}

Hoggar~\cite{Hog74} proved that ordinary polynomial multiplication preserves the LC-NIZ property. The numerator $\W(pq)$, however, generally differs from $\W(p)\W(q)$, and its coefficients involve a binomial kernel instead of ordinary convolution.

Some properties are known to be preserved under this operation.
For example, Wagner~\cite{Wag92} proved that if $\W(p)$ and $\W(q)$ have only nonpositive real zeros, then so does $\W(pq)$. 
Br\"and\'en, Ferroni, and Jochemko~\cite{BFJ24} established the preservation of ultra-log-concavity. 
Here a nonnegative sequence $(a_i)_{i=0}^{d}$ is ultra-log-concave of order $d$ whenever $(a_i/\binom{d}{i})_{i=0}^{d}$ is LC-NIZ. 
They also showed that ultra-log-concavity of $\W(p)$, combined with the LC-NIZ property of $\W(q)$, implies the LC-NIZ property of $\W(pq)$. 
Furthermore, in~\cite[Question 6.1]{BFJ24}, they asked whether the same conclusion holds when both factors are assumed merely to be LC-NIZ.
Our main result gives an affirmative answer to this question.
\begin{theorem}\label{thm:main}
If $\W(p)$ and $\W(q)$ are LC-NIZ, then so is $\W(pq)$.
\end{theorem}

The paper is organized as follows. 
In Section~\ref{sec:Auxres}, we review some results on $\TP$ and $\RR$ matrices and establish the coefficient formula. 
In Section~\ref{sec:Berpoly}, we introduce Bernstein-type operators and show
that the kernel matrix preserves the LC-NIZ property.
In Section~\ref{sec:proofmain}, we prove Theorem~\ref{thm:main}. 
Finally, we present an application to Ehrhart theory in Section~\ref{sec:app}.

\section{Preliminaries}
\label{sec:Auxres}
Let $p\in\mathbb{R}[t]$ be a nonzero polynomial of degree $d$. 
The identity
\begin{equation*}
\sum_{n\geq0}\binom{n+d-i}{d}x^n
=\frac{x^i}{(1-x)^{d+1}},
\quad 0\leq i\leq d,
\end{equation*}
shows that if $\W(p)(x)=\sum_{i=0}^d a_i x^i$, then
\begin{equation}\label{eq:binomial-expansion}
p(t)=\sum_{i=0}^d a_i\binom{t+d-i}{d}.
\end{equation}

The following lemma restates~\cite[Eq. (9)]{FK14} in our notation.
\begin{lemma}\label{coefficients}
Suppose
\begin{equation*}
\W(p)(x)=\sum_{i=0}^{d_1} a_i x^i,\quad
\W(q)(x)=\sum_{j=0}^{d_2} b_j x^j,
\qquad \deg p=d_1,\quad \deg q=d_2.
\end{equation*}
If $\W(pq)(x)=\sum_{k=0}^{d_1+d_2}c_kx^k$, then
\begin{equation}\label{eq:c-original}
c_k=\sum_{i=0}^{d_1}\sum_{j=0}^{d_2}
a_i b_j
\binom{d_1-i+j}{k-i}\binom{d_2-j+i}{k-j}.
\end{equation}
Define the matrix $H^{d_1,d_2}=
\left[ H^{d_1,d_2}_{i,j} \right]_{i,j}$ with
\begin{equation}\label{eq:H}
H^{d_1,d_2}_{i,j}
=\binom{d_1+i-j}{i}\binom{d_2+j-i}{j},
\quad 0\leq i\leq d_2,\ 0\leq j\leq d_1.
\end{equation}
Then $H^{d_1,d_2}$ is a matrix with strictly positive entries and
\begin{equation*}
c_k=\sum_{i=0}^{d_2} \sum_{j=0}^{d_1}
H^{d_1,d_2}_{i,j}\,a_{k-i}b_{k-j}.
\end{equation*}
\end{lemma}


Following Karlin~\cite{Kar68}, a matrix $M$ with nonnegative entries is \emph{totally positive of order two} (or \emph{$\TP$}) if
\begin{equation*}
M_{r_1,s_1}M_{r_2,s_2}
-M_{r_1,s_2}M_{r_2,s_1}\geq0
\quad(r_1<r_2,\ s_1<s_2);
\end{equation*}
and it is \emph{reverse regular of order two} (or \emph{$\RR$}) if 
\begin{equation*}
M_{r_1,s_1}M_{r_2,s_2}
-M_{r_1,s_2}M_{r_2,s_1}\leq0
\quad(r_1<r_2,\ s_1<s_2).
\end{equation*}
We first review some basic facts about $\TP$ and $\RR$ matrices.
Note that a sequence $w$ is LC-NIZ if and only if $w_i w_j\ge w_{i-1} w_{j+1}$ for all $i\le j$.
The first lemma 
follows directly from the definitions,
and the second follows immediately from the classical Cauchy--Binet formula.
\begin{lemma}\label{lem:toeplitz}
A sequence $w$ is LC-NIZ if and only if 
its Toeplitz matrix
$[w_{i-j}]_{i,j}
$
is $\TP$. 
\end{lemma}

\begin{lemma}\label{lem:composition}
The product of two $\TP$ matrices is
$\TP$. 
The product of a $\TP$ matrix and an $\RR$ matrix is $\RR$.
In particular, convolution of finite LC-NIZ sequences
is LC-NIZ.
\end{lemma}

For a nonnegative integer $r$, let
$\fall xr=x(x-1)\cdots(x-r+1)$, with
$\fall x0=1$.

\begin{proposition}\label{lem:H-RR}
The matrix $H^{d_1,d_2}$ is $\RR$.
\end{proposition}

\begin{proof}
If $d_1=0$ or $d_2=0$, the statement holds trivially. 
Let $H=H^{d_1,d_2}$.
For $0\leq i_1<i_2\le d_2$ and $0\leq j_1< j_2\le d_1$, a direct computation gives
\begin{equation*}
\frac{H_{i_1,j_1}H_{i_2,j_2}}{H_{i_1,j_2}H_{i_2,j_1}}
= \frac{\fall {d_1+i_1-j_1}{j_2-j_1}\fall{d_2+j_2-i_2}{j_2-j_1}}{\fall{d_1+i_2-j_1}{j_2-j_1}\fall{d_2+j_2-i_1}{j_2-j_1}} < 1.
\end{equation*}
This proves that $H$ is $\RR$.
\end{proof}
\section{Two matrices arising from Bernstein polynomials}
\label{sec:Berpoly}
Let $A_d=[A_d(i,j)]_{i=0,j=0}^{d+1,d}$ be a $(d+2)\times (d+1)$ matrix with entries 
$$A_d(i,i)=d+1-i, \quad A_d(i+1,i)=i+1$$ and  $A_d(i,j)=0$ otherwise.
Define  
$$E_d=\frac{1}{d+1}A_d, \quad d\ge 0,$$
$$D_d=\frac{1}{d+1}A^{\mathsf T}_{d-1}, \quad d\ge 1.$$
Clearly, for $d\ge 1,$ 
\begin{equation}\label{transpose}
D_d=\frac{d}{d+1}E^{\mathsf T}_{d-1}.
\end{equation}
For a sequence $w=(w_0,\ldots,w_d)$,  we have 
\begin{equation}\label{eq:E}
(E_dw)_i=
\frac{i w_{i-1}+(d+1-i)w_i}{d+1},
\quad 0\leq i\leq d+1,
\end{equation}
\begin{equation}\label{eq:D}
(D_dw)_i=
\frac{(d-i)w_i+(i+1)w_{i+1}}{d+1},
\quad 0\leq i\leq d-1,
\end{equation}
where missing entries are zero.

The matrix $E_d$ is closely related to the Bernstein transform.
Let $\R_d[t]$ be the space of polynomials with real coefficients of degree at most $d$.
The Bernstein polynomials 
\begin{equation*}
B_{d,i}(t)=\binom{d}{i} t^i(1-t)^{d-i}\quad \text{for} \ 0\le i \le d
\end{equation*}
form a basis of $\R_d[t]$
and satisfy the following recurrence
\begin{equation}
\label{eq:recurB}
B_{d,i}(t)
=\frac{d+1-i}{d+1}B_{d+1,i}(t)
+\frac{i+1}{d+1}B_{d+1,i+1}(t).
\end{equation}
We refer the reader to~\cite{Lor86,Pen99} for more details on Bernstein polynomials.

\begin{lemma}\label{lem:rep}
Let $f(t)=\sum_{i=0}^{d}w_iB_{d,i}(t)$.  Then
$$
f(t)=\sum_{i=0}^{d+1}(E_dw)_i B_{d+1,i}(t).
$$
\end{lemma}

\begin{proof}
By~\eqref{eq:recurB}, we have 
\begin{align*}
f(t)&=\sum_{i=0}^{d}w_i\left( \frac{d+1-i}{d+1}B_{d+1,i}(t)
+\frac{i+1}{d+1}B_{d+1,i+1}(t) \right)\\
&=\sum_{i=0}^{d+1}\left( \frac{(d+1-i)w_i}{d+1}+\frac{i\,w_{i-1}}{d+1} \right)B_{d+1,i}(t)\\
&=\sum_{i=0}^{d+1}(E_dw)_i B_{d+1,i}(t).
\end{align*}
This completes the proof.
\end{proof}

\begin{lemma}\label{lem:E&D}
Let $w=(w_0,\ldots,w_d)$ be a sequence of nonnegative numbers.
\begin{itemize}
\item For $n\geq d$,
let 
$$
E^{(n)}:=E_{n-1}\cdot E_{n-2}\cdots E_{d+1}\cdot E_d,
$$
and $E^{(n)}w=w^{(n)}=\{w_i^{(n)}\}_{i=0}^n$.
Then
\begin{equation}\label{eq:elevation-weights}
w_i^{(n)}
=\sum_{k=0}^d
\frac{\binom{d}{k}\binom{n-d}{i-k}}{\binom{n}{i}}\, w_k,
\quad 0\leq i\leq n.
\end{equation}
\item  For $d\geq m$, let 
$$D^{(m)}:=D_{m+1}\cdot D_{m+2}\cdots  D_{d-1}\cdot D_d,$$
and $D^{(m)}w=w^{(m)}=\{w_i^{(m)}\}_{i=0}^m$,
where $D^{(m)}$ is the identity for $d=m$. Then 
\begin{equation}\label{eq:del-weights}
w_i^{(m)}
=\sum_{j=0}^d
\frac{\binom{j}{i}\binom{d-j}{m-i}}{\binom{d+1}{m+1}}\, w_j,
\quad 0\leq i\leq m.
\end{equation}
\end{itemize}
\end{lemma}

\begin{proof}
By Lemma~\ref{lem:rep},
\begin{equation}\label{eq:Iteraf}
f(t)=\sum_{i=0}^{d}w_iB_{d,i}(t)=\sum_{i=0}^{n}w^{(n)}_iB_{n,i}(t).
\end{equation}
For $n\ge d,$ we have 
\begin{align*}
B_{d,i}(t)&=B_{d,i}(t)(t+(1-t))^{n-d}=\binom{d}{i}t^i(1-t)^{d-i}(t+(1-t))^{n-d}\\
&=\binom{d}{i}t^i(1-t)^{d-i}\sum_{j=0}^{n-d}\binom{n-d}{j}t^j(1-t)^{n-d-j}
=\sum_{j=0}^{n-d}\binom{d}{i}\binom{n-d}{j}t^{i+j}(1-t)^{n-(i+j)}.
\end{align*}
Reindexing the sum gives
$$
B_{d,i}(t)=\sum_{j=0}^{n}\binom{d}{i}\binom{n-d}{j-i}t^{j}(1-t)^{n-j}
=\sum_{j=0}^{n}\frac{\binom{d}{i}\binom{n-d}{j-i}}{\binom{n}{j}} \, B_{n,j}(t).
$$
Combining this identity with~\eqref{eq:Iteraf} proves~\eqref{eq:elevation-weights}.

By~\eqref{transpose}, $D^{(m)}$ equals
$$
\frac{m+1}{m+2} E_m^{\mathsf T}\cdot \frac{m+2}{m+3} E_{m+1}^{\mathsf T}\cdots \frac{d}{d+1} E_{d-1}^{\mathsf T}=\frac{m+1}{d+1} E_m^{\mathsf T} E_{m+1}^{\mathsf T}\cdots E_{d-1}^{\mathsf T}
=\frac{m+1}{d+1}\left(E_{d-1} E_{d-2}\cdots E_{m}\right)^{\mathsf T}.
$$
By~\eqref{eq:elevation-weights},
the $(i,j)$-entry of $E_{n-1}\cdot E_{n-2}\cdots E_{d+1}\cdot E_d$
is 
$
{\binom{d}{j}\binom{n-d}{i-j}}/{\binom{n}{i}}
$.
Substituting $n=d, d=m$,
we obtain that the $(i,j)$-entry of $E_{d-1} E_{d-2}\cdots E_{m}$ is
$
{\binom{m}{j}\binom{d-m}{i-j}}/{\binom{d}{i}}.
$
Hence,
$$
[D^{(m)}]_{i,j}=\frac{m+1}{d+1}
\left[ \left(E_{d-1} E_{d-2}\cdots E_{m}\right)^{\mathsf T} \right]_{i,j}
=\frac{m+1}{d+1}\frac{\binom{m}{i}\binom{d-m}{j-i}}{\binom{d}{j}}.
$$
A direct computation yields
\begin{equation*}
\frac{m+1}{d+1}
\frac{\binom{m}{i}\binom{d-m}{j-i}}{\binom{d}{j}}
=
\frac{(m+1)!(d-m)!}{(d+1)!}
\frac{j!}{i!(j-i)!}
\frac{(d-j)!}{(m-i)!(d-m-j+i)!}
=
\frac{\binom{j}{i}\binom{d-j}{m-i}}
{\binom{d+1}{m+1}}.
\end{equation*}
Consequently,
\begin{equation*}
w_i^{(m)}
=
(D^{(m)} w)_i
=
\sum_{j=0}^{d}
\frac{\binom{j}{i}\binom{d-j}{m-i}}
{\binom{d+1}{m+1}}\,w_j,
\quad 0\le i\le m.
\end{equation*}
This completes the proof.
\end{proof}

\begin{remark}
The matrix $E^{(n)}$ in Lemma~\ref{lem:E&D} appeared in~\cite{QRR11}, where it was used to explicitly compute higher-degree Bernstein coefficients and establish their uniform approximation to polynomial values.
\end{remark}

We use the two finite matrices $E_d$ and $D_d$ to construct an
LC-NIZ-preserving transform. 

\begin{lemma}\label{lem:averaging}
Let $w=(w_0,\ldots,w_d)$ be a sequence of nonnegative numbers.
Suppose that $w$ is LC-NIZ.
Then both  $E_dw$ and $D_dw$ are LC-NIZ.
\end{lemma}

\begin{proof}	
Without loss of generality, 
assume that all entries of $w$ are positive.
For $1\le i\le d$,
\begin{align*}
(d+1)(E_dw)_i&=i w_{i-1}+(d+1-i)w_i=w_{i-1}\,\left(i+\frac{(d+1-i)\,w_{i}}{w_{i-1}}\right).
\end{align*}
Since $w_{i-2}\le {w_{i-1}^2}/{w_i}$ and $w_{i+1}\le w_i^2/w_{i-1},$
we have 
\begin{align*}
(d+1)(E_dw)_{i-1}&=(i-1)w_{i-2}+(d+2-i)w_{i-1}\le w_{i-1}\left(\frac{(i-1)w_{i-1}}{w_i}+d+2-i\right),\\
(d+1)(E_dw)_{i+1}&=(i+1)w_{i}+(d-i)w_{i+1}\le w_{i}\left(i+1+\frac{(d-i)\,w_i}{w_{i-1}}\right).
\end{align*}
Thus, 
\begin{align*}
(d+1)^2(E_dw)_{i-1}(E_dw)_{i+1}&\le w_{i-1}w_i\left(\frac{(i-1)w_{i-1}}{w_{i}}+d+2-i\right)\left(i+1+\frac{(d-i)w_i}{w_{i-1}}\right)\\
&=w^2_{i-1}\left((i-1)+\frac{(d+2-i)w_i}{w_{i-1}}\right)\left(i+1+\frac{(d-i)w_i}{w_{i-1}}\right) \\
&\le w^2_{i-1}\left(i+\frac{(d+1-i)w_i}{w_{i-1}}\right)^2=(d+1)^2((E_dw)_i)^2.
\end{align*}
Hence, the sequence $E_dw$ is log-concave. 

Assume that the sequence $E_dw$ has an internal zero, i.e., there exist indices $i,j,k$ satisfying $i<k<j$ such that $(E_dw)_i>0$, $(E_dw)_k=0$, and $(E_dw)_j>0$. 
By~\eqref{eq:E}, if $(E_dw)_k=0$, then $w_{k-1}=w_k=0$. 
Moreover, since $(E_dw)_i>0$ and $(E_dw)_j>0$, it follows that the sequence $w$ has an internal zero, which yields a contradiction. 
Therefore, $E_dw$ is an LC-NIZ sequence.

The proof for $D_dw$ is similar, and we leave it to the interested reader.
\end{proof}

\begin{proposition}\label{prop:integral}
Let $w=(w_0,\ldots,w_d)$ be LC-NIZ. Define
\begin{equation*}
F(t)=\sum_{k=0}^d\binom dk w_k t^k(1-t)^{d-k}.
\end{equation*}
For every integer $m\geq0$, the vector $(u_i)_{0\le i\le m}$ is LC-NIZ,
where 
\begin{equation}\label{eq:integral-transform}
u_i
=(m+1)\binom{m}{i}
\int_0^1 t^i(1-t)^{m-i}F(t)\,dt,
\quad 0\leq i\leq m.
\end{equation}
If $w\neq0$, then $u_i>0$ for all $0\le i\le m$.
\end{proposition}

\begin{proof}
Take $n\geq\max(d,m)$.
Let 
$$
v^{(n)}=D_{m+1}\cdot D_{m+2}\cdots D_{n-1}\cdot D_n \cdot E_{n-1}\cdot E_{n-2}\cdots E_{d+1}\cdot E_d \cdot w.
$$
We show that $v_j^{(n)}$ converges to $u_j$ for $0\le j \le m$.

Let $w^{(n)}=E_{n-1}\cdot E_{n-2}\cdots E_{d+1}\cdot E_d \cdot w.$
Then by~\eqref{eq:elevation-weights},
$$
w_i^{(n)}
=\sum_{k=0}^d
\frac{\binom{d}{k}\binom{n-d}{i-k}}{\binom{n}{i}}\, w_k,
\quad 0\leq i\leq n.
$$
Recall that
$\fall xr=x(x-1)\cdots(x-r+1)$ with
$\fall x0=1$. Then
\begin{equation}\label{eq:falling-ratio}
w_i^{(n)}
=\sum_{k=0}^d
\binom dk
\frac{\fall ik\fall{n-d}{i-k}}{\fall ni}\, w_k=\sum_{k=0}^d
\binom dk
\frac{\fall ik\fall{n-i}{d-k}}{\fall nd}\, w_k.
\end{equation}
Note that 
$$
\frac{\fall ik\fall{n-i}{d-k}}{\fall nd}=\frac{\left(n^{-k} \fall ik\right) \left(n^{-d+k} \fall{n-i}{d-k}\right)}{n^{-d}\fall nd}.
$$
Recall that 
$$
F(i/n)=\sum_{k=0}^d\binom dk w_k (i/n)^k(1-i/n)^{d-k}.
$$
For each fixed $k$,
\begin{equation*}
\sup_{0\leq i\leq n}
\left|n^{-k}\fall ik-(i/n)^k\right|\xrightarrow{n\to\infty} 0,\quad  
\sup_{0\leq i\leq n}
\left|n^{-d+k}\fall {n-i}{d-k}-(1-i/n)^{d-k}\right|\xrightarrow{n\to\infty} 0
\end{equation*}
since $\fall xr$ is a polynomial in $x$ with
leading term $x^r$.
Hence,
\begin{equation}\label{eq:uniform-elevation}
\sup_{0\leq i\leq n}
|w_i^{(n)}-F(i/n)|\xrightarrow{n\to\infty} 0.
\end{equation}

By~\eqref{eq:del-weights},
\begin{equation}\label{eq:v}
v_i^{(n)}=\sum_{j=0}^n
\frac{\binom{j}{i}\binom{n-j}{m-i}}{\binom{n+1}{m+1}}\, w^{(n)}_j,\quad 0\leq i\leq m.
\end{equation}
Similarly, 
\begin{align*}
n \frac{\binom{j}{i}\binom{n-j}{m-i}}{\binom{n+1}{m+1}}
&=\frac{n(m+1)!}{i!(m-i)!}
\frac{\fall ji\fall{n-j}{m-i}}
{\fall{n+1}{m+1}}\\
&=n(m+1)\binom mi\frac{\fall ji\fall{n-j}{m-i}}
{\fall{n+1}{m+1}}\\
&=(m+1)\binom mi\frac{[n^{-i}\fall ji] [n^{-m+i}\fall{n-j}{m-i}]}
{n^{-m-1}\fall{n+1}{m+1}}.
\end{align*}
For each fixed $i\in\{0,\ldots,m\}$, 
\begin{equation}\label{eq:uniform-reduction}
\sup_{0\leq j\leq n}
\left|
n \frac{\binom{j}{i}\binom{n-j}{m-i}}{\binom{n+1}{m+1}}
-(m+1)\binom mi(j/n)^i(1-j/n)^{m-i}
\right|\xrightarrow{n\to\infty} 0.
\end{equation}

For $d\ge 1$, let $z=(1,1,\ldots,1)$.
By the definition of $D_d$,
\begin{equation*}
(D_dz)_i=\frac{(d-i)\cdot 1+(i+1)\cdot 1}{d+1}=1,
\quad 0\le i\le d-1.
\end{equation*}
Thus, $D_d$ maps the all-ones vector of length $d+1$ to
the all-ones vector of length $d$. 
Consequently,
$D_{m+1}\cdots D_n$ maps the all-ones vector of length $n+1$
to the all-ones vector of length $m+1$.
By~\eqref{eq:del-weights}, we obtain
\begin{equation*}
\sum_{j=0}^{n}
\frac{\binom{j}{i}\binom{n-j}{m-i}}
{\binom{n+1}{m+1}}
=1,
\quad 0\le i\le m.
\end{equation*}
Fix $i\in\{0,\ldots,m\}$. By~\eqref{eq:uniform-elevation} and~\eqref{eq:v}, 
\begin{equation*}
\left|
v_i^{(n)}-
\sum_{j=0}^{n}
\frac{\binom{j}{i}\binom{n-j}{m-i}}
{\binom{n+1}{m+1}}
F(j/n)
\right|
\, \le
\sum_{j=0}^{n}
\frac{\binom{j}{i}\binom{n-j}{m-i}}
{\binom{n+1}{m+1}}
\left|w_j^{(n)}-F(j/n)\right|
\, \le
\max_{0\le j\le n}
\left|w_j^{(n)}-F(j/n)\right|
\xrightarrow{n\to\infty} 0.
\end{equation*}
Moreover, $F(t)$ is bounded on $[0,1]$. Thus, by~\eqref{eq:uniform-reduction}, 
\begin{align*}
&\Biggl|
\sum_{j=0}^{n}
\frac{\binom{j}{i}\binom{n-j}{m-i}}
{\binom{n+1}{m+1}}
F(j/n)-
\frac{m+1}{n}\binom{m}{i}
\sum_{j=0}^{n}
(j/n)^i(1-j/n)^{m-i}F(j/n)
\Biggr|\\
&\quad\le
\frac{n+1}{n}\max_{0\le t\le 1}|F(t)|\cdot
\sup_{0\le j\le n}
\left|
n\frac{\binom{j}{i}\binom{n-j}{m-i}}
{\binom{n+1}{m+1}}
-
(m+1)\binom{m}{i}(j/n)^i(1-j/n)^{m-i}
\right|
\xrightarrow{n\to\infty} 0.
\end{align*}

Since $t^i(1-t)^{m-i}F(t)$ is continuous on $[0,1]$,
Riemann-sum convergence yields
\begin{equation*}
\frac{m+1}{n}\binom{m}{i}
\sum_{j=0}^{n}
(j/n)^i(1-j/n)^{m-i}F(j/n)
\xrightarrow{n\to\infty}
(m+1)\binom{m}{i}
\int_0^1 t^i(1-t)^{m-i}F(t)\,dt
=u_i.
\end{equation*}
Combining the preceding estimates, we obtain
\begin{equation*}
v_i^{(n)} \xrightarrow{n\to\infty} u_i,
\quad 0\le i\le m.
\end{equation*}

By
Lemma~\ref{lem:averaging}, $v^{(n)}$ is log-concave with no internal zeros for every $n$.  
In particular,
\begin{equation*}
\left(v_i^{(n)}\right)^2\ge v_{i-1}^{(n)}v_{i+1}^{(n)},
\quad 1\le i\le m-1.
\end{equation*}
Letting $n\to\infty$, we obtain
\begin{equation*}
u_i^2\ge u_{i-1}u_{i+1},
\quad 1\le i\le m-1.
\end{equation*}
Thus $(u_0,\ldots,u_m)$ is nonnegative and log-concave.
It remains to verify that there are no internal zeros.
If $w\ne 0$, then at least one $w_k$ is positive, while all
$w_k$ are nonnegative. Consequently,
\begin{equation*}
F(t)=
\sum_{k=0}^{d}\binom{d}{k}w_k t^k(1-t)^{d-k}>0,
\quad \text{for}\ 0<t<1.
\end{equation*}
Since $t^i(1-t)^{m-i}>0$ on $(0,1)$, it follows from~\eqref{eq:integral-transform}
that
\begin{equation*}
u_i
=
(m+1)\binom{m}{i}
\int_0^1 t^i(1-t)^{m-i}F(t)\,dt
>0,
\quad 0\le i\le m.
\end{equation*}
Hence the output has no internal zeros. If $w=0$, then
$F=0$ and $u_i=0$ for every $i$, so the conclusion is immediate.
This completes the proof.
\end{proof}

\begin{corollary}\label{coro:kernel-LC}
For $H=H^{d_1,d_2}$ in \eqref{eq:H}, both $H$ and
$H^\transpose$ map LC-NIZ vectors to LC-NIZ vectors.
The image of any nonzero nonnegative vector is strictly
positive.
\end{corollary}

\begin{proof}
Let $w=(w_0,\ldots,w_{d_1})$ be LC-NIZ and let
\begin{equation*}
F(t)=\sum_{k=0}^{d_1} w_k B_{d_1,k}(t)=\sum_{k=0}^{d_1}\binom{d_1}{k} w_k t^{k}(1-t)^{d_1-k},
\end{equation*}
and
\begin{equation*}
\widetilde F(t)=F(1-t)=\sum_{k=0}^{d_1}\binom{{d_1}}{k} w_k t^{{d_1}-k}(1-t)^k
=\sum_{k=0}^{d_1}\binom{{d_1}}{k} w_{d_1-k} t^{k}(1-t)^{d_1-k}.
\end{equation*}

The beta integral
$\int_0^1t^\alpha(1-t)^\beta\,dt
=\alpha!\beta!/(\alpha+\beta+1)!$, for
nonnegative integers $\alpha,\beta$, gives
\begin{equation*}
H_{i,j}
=\kappa_{d_1,d_2}\binom{d_2}{i}\binom{d_1}{j}
\int_0^1
t^{d_1+i-j}(1-t)^{d_2+j-i}\,dt,
\end{equation*}
where $\kappa_{d_1,d_2}:= \frac{(d_1+d_2+1)!}{d_1!\,d_2!}$. 
Then we have
\begin{align*}
(Hw)_i&=\sum_{k=0}^{d_1}H_{i,k}w_k=\sum_{k=0}^{d_1}w_k \kappa_{d_1,d_2}\binom{d_2}{i}\binom{d_1}{k}
\int_0^1
t^{d_1+i-k}(1-t)^{d_2+k-i}\,dt \\
&=\frac{\kappa_{d_1,d_2}}{d_2+1}\left( (d_2+1)\binom{d_2}{i}\int_0^1 t^i(1-t)^{d_2-i}\widetilde F(t)\,dt\right),
\end{align*}
for any $0\le i \le d_2$.
Since the sequence $(w_{d_1},\ldots,w_{0})$ is LC-NIZ,
Proposition~\ref{prop:integral} proves that $Hw$ is LC-NIZ.
All entries of $H$ are positive, so a nonzero
nonnegative input has a positive image regardless of
log-concavity.
Finally,
\begin{equation*}
(H^{d_1,d_2})^\transpose=H^{d_2,d_1},
\end{equation*}
which proves the statement for the transpose.
\end{proof}

\section{Proof of the main theorem}
\label{sec:proofmain}
In this section, we give the proof of Theorem~\ref {thm:main}.
First, we introduce the following result, which plays a key role in the proof.
\begin{lemma}\label{lem:diagonal}
Let $C=[C_{i,j}]_{i,j\ge 0}$ be a nonnegative matrix. 
Suppose that 
\begin{enumerate}
\item[{\rm(C1)}] $C$ is $\RR$,
\item[{\rm(C2)}] all rows and columns of $C$ are log-concave,
\item[{\rm(C3)}] the positive support $\supp((C_{k,k})_{k\ge0})$ is an integer interval (possibly empty).
\end{enumerate} 
Then $(C_{k,k})_{k\ge0}$ is LC-NIZ.
\end{lemma}

\begin{proof}
By {\rm(C3)}, it suffices to prove that $(C_{k,k})_{k\ge0}$ is log-concave.
Assume that $C_{k-1,k-1}C_{k+1,k+1}>0$.
We first show that $C_{k+1,k-1}$ is strictly positive. 
By {\rm(C1)}, 
\begin{equation}\label{eq:diag-pivot}
C_{k-1,k-1}C_{k+1,k+1} \leq C_{k-1,k+1}C_{k+1,k-1}.
\end{equation}
Since the left-hand side of \eqref{eq:diag-pivot} is strictly positive, the right-hand side is also strictly positive, which immediately implies $C_{k+1,k-1}>0$.

By {\rm(C2)}, we have
\begin{align}
C_{k,k-1}^2
&\geq C_{k-1,k-1}C_{k+1,k-1},
\label{eq:diag-column}\\
C_{k+1,k}^2
&\geq C_{k+1,k-1}C_{k+1,k+1}.
\label{eq:diag-row}
\end{align}
Multiplying \eqref{eq:diag-column} and \eqref{eq:diag-row} yields
\begin{equation}\label{eq:diag-prod}
(C_{k,k-1}C_{k+1,k})^2 \geq C_{k-1,k-1}C_{k+1,k+1}C_{k+1,k-1}^2.
\end{equation}
On the other hand, since $C$ is $\mathrm{RR}_2$, we have 
\begin{equation}\label{eq:diag-RR}
C_{k,k-1}C_{k+1,k} \leq C_{k,k}C_{k+1,k-1}.
\end{equation}
By~\eqref{eq:diag-prod} and~\eqref{eq:diag-RR}, we obtain that
\begin{equation*}
C_{k-1,k-1}C_{k+1,k+1}C_{k+1,k-1}^2 \leq (C_{k,k-1}C_{k+1,k})^2 \leq C_{k,k}^2 C_{k+1,k-1}^2.
\end{equation*}
Since $C_{k+1,k-1} > 0$, dividing both sides by the strictly positive factor $C_{k+1,k-1}^2$ yields
\begin{equation*}
C_{k,k}^2 \geq C_{k-1,k-1}C_{k+1,k+1},
\end{equation*}
which completes the proof.
\end{proof}

\begin{proof}[Proof of Theorem~\ref{thm:main}]
Write $d_1=\deg p$, $d_2=\deg q$.
Let $(a_i)$ and $(b_j)$
be the LC-NIZ coefficient sequences of $\W(p)$ and $\W(q)$, respectively, extended by zero outside their index ranges.
Set $H=H^{d_1,d_2}$ and define the matrix $C=[C_{r,s}]_{0\le r,s \le d_1+d_2}$ with entries
\begin{equation}\label{eq:C}
C_{r,s}=\sum_{i=0}^{d_2}\sum_{j=0}^{d_1}
H_{i,j}a_{r-i}b_{s-j}.
\end{equation}
By Lemma~\ref{coefficients}, the diagonal sequence $(C_{k,k})_{0\le k \le d_1+d_2}$ of $C$ is the coefficient sequence of $\W(pq)$, where
\begin{equation}\label{eq:Ck}
C_{k,k}=\sum_{i=0}^{d_2} \sum_{j=0}^{d_1}
H_{i,j}a_{k-i}b_{k-j}.
\end{equation}
By Lemma~\ref{lem:diagonal}, it suffices to prove that the matrix $C$ satisfies conditions {\rm(C1)}, {\rm(C2)} and {\rm(C3)} therein.

\medskip
\noindent
\textbf{Condition {\rm(C1)}.}
Let $T_a$ denote the Toeplitz matrix $[a_{r-i}]$ with rows $0\le r \le d_1+d_2$ and columns $0\leq i\leq d_2$, and let $T_b$ denote $[b_{s-j}]$ with rows $0\le s \le d_1+d_2$ and columns $0\leq j\leq d_1$.
Equation~\eqref{eq:C} is equivalent to the matrix factorization
\begin{equation*}
C=T_a H T_b^\transpose.
\end{equation*}
By Lemma~\ref{lem:toeplitz}, $T_a$ and $T_b$ are $\TP$, and transposition preserves this property. Since $H$ is $\RR$ by Proposition~\ref{lem:H-RR}, Lemma~\ref{lem:composition} implies that $C$ is $\RR$.

\medskip
\noindent
\textbf{Condition {\rm(C2)}.}
Fix $s \in \{0, \dots, d_1+d_2\}$. The $s$-th column of $C$ is given by
\begin{equation*}
C_{r,s}=\sum_{i=0}^{d_2}a_{r-i}\left(H \beta^{(s)}\right)_i \quad \text{for} \ 0\le r \le d_1+d_2,
\end{equation*}
where the column vector $\beta^{(s)}=(b_{s-j})_{j=0}^{d_1}$ is obtained by reversing and restricting the sequence $(b_j)$. 
Clearly, $\beta^{(s)}$ is either LC-NIZ or identically zero.
By Corollary~\ref{coro:kernel-LC} and Lemma~\ref{lem:composition}, the column sequence $(C_{r,s})_{0\le r \le d_1+d_2}$ is LC-NIZ. 
Similarly, for each fixed $r\in \{0, \dots, d_1+d_2\}$, the $r$-th row of $C$ is LC-NIZ.

\medskip
\noindent
\textbf{Condition {\rm(C3)}.}
Since every entry $H_{i,j}$ is strictly positive and the sequences $a$ and $b$ are nonnegative, 
\eqref{eq:Ck} implies that
$C_{k,k}>0$ if and only if there exists a pair of indices $(i, j) \in [0, d_2] \times [0, d_1]$ such that
\begin{equation*}
a_{k-i} > 0 \quad \text{and} \quad b_{k-j} > 0.
\end{equation*}
Let $\supp(a)=[L_a,U_a]$ and $\supp(b)=[L_b,U_b]$.
By definition of the positive supports, this condition is equivalent to the simultaneous existence of $i \in [0, d_2]$ and $j \in [0, d_1]$ satisfying
\begin{equation*}
k-i \in [L_a, U_a] \quad \text{and} \quad k-j \in [L_b, U_b]. 
\end{equation*}
Such an index $i$ exists if and only if $k \in [L_a, U_a + d_2]$, and such an index $j$ exists if and only if $k \in [L_b, U_b + d_1]$. Because the choices of $i$ and $j$ are independent, a valid pair $(i, j)$ exists if and only if
\begin{equation*}
k \in [L_a, U_a + d_2] \cap [L_b, U_b + d_1] = [\max(L_a, L_b), \, \min(U_a + d_2, U_b + d_1)].
\end{equation*}
Since $L_a\leq d_1\leq U_b+d_1$ and $L_b\leq d_2\leq U_a+d_2$, this intersection is a nonempty integer interval.
Hence, the positive support of the diagonal sequence $(C_{k,k})_{0\le k \le d_1+d_2}$ is an integer interval.

Consequently, Lemma~\ref{lem:diagonal} implies that $(C_{k,k})$ is LC-NIZ, which completes the proof.
\end{proof}

\begin{remark}
In Theorem~\ref{thm:main}, the assumption that the coefficient sequences have no internal zeros is essential.
Consider the polynomials
\begin{equation*}
p(t)=\binom{t+3}{3}+\binom{t}{3},\quad q(t)=t.
\end{equation*}
Then $\W(p)(x)=1+x^3$ and $\W(q)(x)=x$.
Both of these polynomials are log-concave, but $\W(p)(x)$ has internal zeros.
However, we have
\begin{equation*}
\W(pq)(x)=4x+3x^3+x^4,
\end{equation*}
which is not log-concave.
\end{remark}
\section{Applications}
\label{sec:app}
In this section, we present an application of Theorem~\ref{thm:main} to Ehrhart theory.
A \emph{lattice polytope} is the convex hull of finitely many points in $\mathbb Z^d$. Let $P\subseteq\mathbb R^d$ be a lattice polytope of dimension $d$, and put
\begin{equation*}
L_P(n)=\#(nP\cap\mathbb Z^d),\quad n\in\mathbb Z_{\geq0}.	
\end{equation*}
Ehrhart's theorem states that $L_P(n)$ agrees with a polynomial in $n$ of degree $d$. Its generating function, called the \emph{Ehrhart series}, has the form
\begin{equation*}
\operatorname{Ehr}_P(x)=\sum_{n\geq0}L_P(n)x^n
=\frac{h_P^*(x)}{(1-x)^{d+1}}.
\end{equation*}
The numerator is called the~\emph{$h^*$-polynomial} of $P$.
Thus $h_P^*=\W(L_P)$. 
We refer the reader to~\cite{BR15} for more details on Ehrhart theory. 

Let $P\subseteq\mathbb R^d$ and $Q\subseteq\mathbb R^e$ have dimensions $d$ and $e$. Their Cartesian product is a lattice polytope in the product lattice $\mathbb Z^d\times\mathbb Z^e$. For every $n\geq0$, we have
\begin{equation*}
L_{P \times Q}(n)=\# \left((nP \cap \mathbb{Z}^d) \times (nQ \cap \mathbb{Z}^e)\right)=L_P(n) L_Q(n),\quad
\dim(P\times Q)=d+e.
\end{equation*}
The Ehrhart series of $P\times Q$ is thus the Hadamard product of the two Ehrhart series. In particular,
\begin{equation*}
h_{P\times Q}^*=\W(L_PL_Q).
\end{equation*}

Ferroni and Higashitani~\cite{FH24} posed the following question: Let $P$ and $Q$ be two lattice polytopes. If $h_P^*(x)$ and $h_Q^*(x)$ are LC-NIZ, is  $h_{P\times Q}^*(x)$ necessarily LC-NIZ? By Theorem~\ref{thm:main}, we answer this question affirmatively.
\begin{corollary}\label{cor:cartesian-products}
If the $h^*$-polynomials of lattice polytopes $P$ and $Q$ are LC-NIZ, then so is $h_{P\times Q}^*$. 
More generally, this conclusion holds for $P_1\times\cdots\times P_m$ whenever each $h_{P_i}^*$ is LC-NIZ.
\end{corollary}



\begin{thebibliography}{99}
\bibitem{AF11}
J.-P. Allouche, M. Mend\`es~France, \textit{Hadamard grade of power series}, J. Number Theory {\bf 131} (2011), no.~11, 2013--2022.

\bibitem{BR15}
M. Beck, S. Robins, \textit{Computing the continuous discretely}, second edition, Undergraduate Texts in Mathematics, Springer, New York, 2015.

\bibitem{BFJ24}
P. Br{\"a}nd{\'e}n, L. Ferroni, K. Jochemko, \textit{Preservation of inequalities under Hadamard products}, Trans. Amer. Math. Soc., to appear,
arXiv:2408.12386v2.

\bibitem{FH24}
L. Ferroni, A. Higashitani, \textit{Examples and counterexamples in Ehrhart theory}, EMS Surv. Math. Sci. (2024).

\bibitem{FK14}
I. Fischer, M. Juhnke-Kubitzke, \textit{Spectra and eigenvectors of the Segre transformation}, Adv. in Appl. Math. {\bf 56} (2014), 1--19.

\bibitem{Hog74}
S.G. Hoggar, \textit{Chromatic polynomials and logarithmic concavity}, J. Combin. Theory Ser. B {\bf 16} (1974), 248--254.

\bibitem{Kar68}
S. Karlin, \textit{Total positivity. Vol. I}, Stanford Univ. Press, Stanford, CA, 1968.

\bibitem{Lor86}
G.G. Lorentz, \textit{Bernstein polynomials}, second edition, 
Chelsea, New York, 1986.


\bibitem{Pen99}
J.M. Pe\~na (Ed.), \textit{Shape preserving representations in computer-aided geometric design}, Nova Science Publishers, Commack, NY, 1999.

\bibitem{QRR11}
W. Qian, M.D. Riedel, I.G. Rosenberg, \textit{Uniform approximation and Bernstein polynomials with coefficients in the unit interval}, European J. Combin. {\bf 32} (2011), no.~3, 448--463.

\bibitem{Sta12}
R.P. Stanley, \textit{Enumerative combinatorics. Volume 1}, second edition, Cambridge Univ. Press, Cambridge, 2012.


\bibitem{Wag92}
D.G. Wagner, \textit{Total positivity of Hadamard products}, J. Math. Anal. Appl. {\bf 163} (1992), no.~2, 459--483.
\end{thebibliography}
\end{document}